\documentclass[11pt,reqno]{amsart}

\usepackage{hyperref}
\usepackage{graphicx}
\usepackage{color}
\usepackage[all]{xy}
\usepackage{amsfonts,amssymb}
\usepackage{bbm} 

\usepackage{mathrsfs}

\usepackage{a4wide}

\newtheorem{neu}{}[section]
\newtheorem{Cor}[neu]{Corollary}
\newtheorem*{Cor*}{Corollary}
\newtheorem{Thm}[neu]{Theorem}
\newtheorem*{Thm*}{Theorem}

\newtheorem{Prop}[neu]{Proposition}
\newtheorem*{Prop*}{Proposition}
\theoremstyle{definition}

\newtheorem*{FunLemma}{Fundamental Lemma}
\newtheorem*{Rmk*}{Remark}
\newtheorem{Rmk}[neu]{Remark}

\newtheorem*{Ex*}{Example}

\newtheorem*{Qu*}{Question}

\newtheorem{Def}[neu]{Definition}

\newcommand{\Z}{\mathbb{Z}}
\newcommand{\R}{\mathbb{R}}
\newcommand{\C}{\mathbb{C}}

\newcommand{\pf}{\longrightarrow}

\newcommand{\id}{\mathrm{id}}

\newcommand{\om}{\omega}
\newcommand{\Om}{\Omega}

\newcommand{\A}{\mathcal{A}}

\newcommand{\D}{\mathbb{D}}

\renewcommand{\L}{\mathscr{L}}

\renewcommand{\H}{\mathrm{H}}

\newcommand{\Ham}{\mathrm{Ham}}

\newcommand{\SH}{\mathrm{SH}}

\newcommand{\RFH}{\mathrm{RFH}}
\newcommand{\RFHb}{\mathbf{RFH}}
\newcommand{\RFC}{\mathrm{RFC}}
\newcommand{\RFCb}{\mathbf{RFC}}

\newcommand{\HM}{\mathrm{HM}}

\newcommand{\Crit}{\mathrm{Crit}}

\newcommand{\beq}{\begin{equation}}
\newcommand{\beqn}{\begin{equation}\nonumber}
\newcommand{\eeq}{\end{equation}}

\newcommand{\bea}{\begin{equation}\begin{aligned}}
\newcommand{\bean}{\begin{equation}\begin{aligned}\nonumber}
\newcommand{\eea}{\end{aligned}\end{equation}}

\numberwithin{equation}{section}

\definecolor{Urs}{rgb}{0,.7,0}
\definecolor{Peter}{rgb}{0,0,1}
\definecolor{red}{rgb}{1,0,0}

\newcommand{\p}{\partial}
\newcommand{\Mp}{\mathfrak{M}}

\begin{document}
\title{Rabinowitz Floer homology: A survey}
\author{Peter Albers}
\author{Urs Frauenfelder}
\address{
    Peter Albers\\
    Department of Mathematics\\
    Purdue University}
\email{palbers@math.purdue.edu}
\address{
    Urs Frauenfelder\\
    Department of Mathematics and Research Institute of Mathematics\\
    Seoul National University}
\email{frauenf@snu.ac.kr}
\begin{abstract}
Rabinowitz Floer homology is the semi-infinite dimensional Morse homology associated to the Rabinowitz action functional used in the pioneering work of Rabinowitz. Gradient flow lines are solutions of a vortex-like equation. In this survey article we describe the construction of Rabinowitz Floer homology and its applications to symplectic and contact topology, global Hamiltonian perturbations and the study of magnetic fields.
\end{abstract}
\maketitle
\tableofcontents

\section{Rabinowitz Floer homology}

\subsection{The Rabinowitz action functional}

In his pioneering work  \cite{Rabinowitz_Periodic_solutions_of_Hamiltonian_systems,Rabinowitz_Periodic_solutions_of_Hamiltonian_systems_on_a_prescribed_enery_surface} Rabinowitz used the Rabinowitz action functional to prove existence of periodic orbits on star-shaped hypersurfaces in $\R^{2n}$. This fundamental work was one of the motivations for Weinstein to state his famous conjecture on the existence of Reeb orbits, \cite{Weinstein_The_conjecture}.

\subsubsection{The unperturbed Rabinowitz action functional and Reeb dynamics} 

\subsubsection*{\underline{Exact symplectic manifolds}}

Let $(M,\om=d\lambda)$ be an exact symplectic manifold, for example $(\R^{2n},\om_0)$ or a cotangent bundle $(T^*B,\om_{std})$ each with its canonical symplectic form. We fix an autonomous Hamiltonian, i.e.~a smooth time-independent function $F:M\to\R$. The Hamiltonian vector field $X_F$ of $F$ is defined implicitly by
\beq
\iota_{X_F}\om=dF\;.
\eeq
Since $F$ is autonomous the Hamiltonian vector field $X_F$ is tangent to level sets of $F$ and therefore its flow $\phi_F^t:M\pf M$ leaves level sets invariant. This means that the energy $F$ is preserved under the flow $\phi_F^t$.

Let $\L:=C^\infty(S^1,M)$, $S^1=\R/\Z$, be the loop space of $M$. The Rabinowitz action functional is defined as follows:
\bea
\A^F:\L\times\R&\pf\R\\
(v,\eta)&\mapsto \A^F(v,\eta):=-\int_{S^1}v^*\lambda-\eta\int_0^1F\big(v(t)\big)dt\;.
\eea
The real number $\eta$ can be thought of as a Lagrange multiplier. Hence critical points $(v,\eta)\in\Crit\A^F$ are critical points of the area functional restricted to the space of loops with $F$-mean value zero. They are solutions of
\beq
\left\{
\begin{aligned}
&\dot{v}(t)=\eta X_F\big(v(t)\big),\;\forall t\in S^1\\
&\int_0^1F\big(v(t)\big)dt=0\;.
\end{aligned}
\right.
\eeq
The first equation can be integrated to $v(t)=\phi_F^{\eta t}(v(0))$ and thus, by preservation of energy, the critical point equation is equivalent to 
\beq
\left\{
\begin{aligned}
&\dot{v}(t)=\eta X_F\big(v(t)\big),&&\forall t\in S^1\\
&v(t)\in F^{-1}(0),&&\forall t\in S^1\;.
\end{aligned}
\right.
\eeq
Hence, critical points of $\A^F$ correspond to periodic orbits of $X_F$ with period $\eta$ and lie on the energy hypersurface $\Sigma:=F^{-1}(0)$. Here, the period $\eta$ is understood in a generalized sense; it is allowed to be negative in which case the periodic orbit is traversed backwards. Moreover, if the period is zero then $v$ is constant and corresponds to a point on the energy hypersurface $F^{-1}(0)$.

If $\Sigma$ is a regular hypersurface for two functions $F,\widetilde{F}:M\to\R$, $\Sigma=\widetilde{F}^{-1}(0)=F^{-1}(0)$, then critical points of $\A^{\widetilde{F}}$ agree up to reparametrization with critical points of $\A^F$. In fact, they are closed characteristics of the canonical line bundle $\ker\om|_\Sigma\to\Sigma$ or constant.

It is interesting to compare the critical points of the Rabinowitz action functional $\A^F$ to critical points of the action functional $\A_F$ of classical mechanics 
\bea
\A_F:\L&\pf\R\\
v&\mapsto \A_F(v):=-\int_{S^1}v^*\lambda-\int_0^1F\big(v(t)\big)dt\;.
\eea
Critical points $v\in\Crit\A_F$ are 1-periodic solutions of
\beq
\dot{v}(t)= X_F\big(v(t)\big)\;.
\eeq
In this case the period of $v$ is fixed but the energy is arbitrary.

\subsubsection*{\underline{Symplectically aspherical manifolds}}
A symplectic manifold $(M,\om)$ is called symplectically aspherical if the homomorphism $I_\om:\pi_2(M)\to\R$ obtained by integrating $\om$ over a smooth representative vanishes identically. This holds for example if $(M,\om=d\lambda)$ is an exact symplectic manifold. In this situation the Rabinowitz action functional can still be defined on the component $\L_0\subset\L$ of the loop space of contractible loops. For $v\in\L_0$ there exists a filling disk $\bar{v}:\D^2\to M$ where $\D^2=\{z\in\C\mid|z|\leq1\}$ and $\bar{v}(e^{2\pi it})=v(t)$. In this case we set
\bea
\A^F:\L_0\times\R&\pf\R\\
(v,\eta)&\mapsto \A^F(v,\eta):=-\int_{\D^2}\bar{v}^*\om-\eta\int_0^1F\big(v(t)\big)dt\;.
\eea
Due to the symplectic asphericity the definition of $\A^F$ does not depend on the choice of the filling disk. If in addition $M$ is symplectically atoroidal, i.e.~$\int_{T^2}f^*\om=0$, $\forall f:T^2\to M$, the Rabinowitz action functional $\A^F$ can be extended to the whole loop space $\L$. An interesting class of symplectically atoroidal manifolds are certain twisted cotangent bundles. A twisted cotangent bundle is $(T^*B,\om_{std}+\tau^*\sigma)$, where $\tau:T^*B\to B$ is the projection and such that $\sigma\in\Om^2(B)$ is closed. If the pull-back of $\sigma$ to the universal cover $\widetilde{B}$ has a bounded primitive then $(T^*B,\om_{std}+\tau^*\sigma)$ is symplectically atoroidal. This fact was used by Merry in \cite{Merry_On_the_RFH_of_twisted_cotangent_bundles}.

\subsubsection{The perturbed Rabinowitz action functional and global Hamiltonian perturbations} \label{sec:perturbed_Rabinwitz_action_functional}

Since $F:M\to\R$ is autonomous the energy hypersurface $\Sigma=F^{-1}(0)$ is preserved under $\phi_F^t$. Therefore, $\Sigma$ is foliated by leaves $L_x:=\{\phi_F^t(x)\mid t\in\R\}$, $x\in \Sigma$. It is a challenging problem to compare the system $F$ before and after a global perturbation occurring in the time interval $[0,1]$. Such a perturbation is described by a function $H:M\times[0,1]\to\R$. Moser observed in \cite{Moser_A_fixed_point_theorem_in_symplectic_geometry} that it is not possible to destroy all trajectories of the unperturbed system if the perturbation is sufficiently small, that is, there exists $x\in \Sigma$ 
\beq
\phi_H^1(x)\in L_x\;.
\eeq
Such a point $x$ is referred to as a leaf-wise intersection point. Equivalently, there exists $(x,\eta)\in \Sigma\times\R$ such that
\beq
\phi_F^{\eta}(x)=\phi_{H}^1(x)\;.
\eeq
We point out that the time shift $\eta$ is uniquely defined by the above equation unless the leaf $L_x$ is closed. If the time shift is negative then the perturbation moves the system back into its own past. Likewise, if the time shift is positive the perturbation moves the system forward into its own future.

Already the existence problem for leaf-wise intersection points is highly non-trivial. The search for leaf-wise intersection points was initiated by Moser in \cite{Moser_A_fixed_point_theorem_in_symplectic_geometry} and pursued further in 
\cite{Banyaga_On_fixed_points_of_symplectic_maps,
Hofer_On_the_topological_properties_of_symplectic_maps,
Ekeland_Hofer_Two_symplectic_fixed_point_theorems_with_applications_to_Hamiltonian_dynamics,
Ginzburg_Coisotropic_intersections,
Dragnev_Symplectic_rigidity_symplectic_fixed_points_and_global_perturbations_of_Hamiltonian_systems,
Albers_Frauenfelder_Leafwise_intersections_and_RFH,
Ziltener_coisotropic, Albers_Frauenfelder_Leafwise_Intersections_Are_Generically_Morse,
Albers_McLean_SH_and_infinitley_many_LI,
Gurel_leafwise_coisotropic_intersection,Kang_Existence_of_leafwise_intersection_points_in_the_unrestricted_case, Albers_Frauenfelder_Remark_on_a_Thm_by_Ekeland_Hofer,
Albers_Frauenfelder_Spectral_invariants_in_RFH, 
Merry_On_the_RFH_of_twisted_cotangent_bundles}. We refer to \cite{Albers_Frauenfelder_Leafwise_Intersections_Are_Generically_Morse} for a brief history.

In \cite{Albers_Frauenfelder_Leafwise_intersections_and_RFH} Albers -- Frauenfelder developed a variational approach to the study of leaf-wise intersection points. 

\begin{Def}\label{def:Moser_pair}
A pair $\Mp=(F,H)$ of Hamiltonians $F,H:M\times S^1\pf R$ is called a Moser pair if it satisfies
\beq
F(\cdot,t)=0\quad\forall t\in[\tfrac12,1]\qquad\text{and}\qquad H(\cdot,t)=0\quad\forall t\in[0,\tfrac12]\;,
\eeq
and $F$ is of the form $F(x,t)=\rho(t)f(x)$ for some smooth map $\rho:S^1\to [0,1]$ with $\int_0^1\rho(t) dt=1$ and $f:M\pf\R$. 
\end{Def}
If we start with an autonomous Hamiltonian $\widehat{F}:M\to\R$ and an arbitrary $\widehat{H}:M\times S^1\to\R$ we can find $F,H:M\times S^1\to\R$ such that the Hamiltonian flows $\phi_F^t,\phi_H^t$ are time reparametrizations of the flows $\phi_{\widehat{F}}^t,\phi_{\widehat{H}}^t$ and such that $(F,H)$ is a Moser pair.

For simplicity we assume that $(M,\om=d\lambda)$ is an exact symplectic manifold. For a Moser pair $\Mp=(F,H)$ the perturbed Rabinowitz action functional is defined by
\bea
\A^{(F,H)}\equiv\A^\Mp:\L\times\R&\pf\R\\
(v,\eta)&\mapsto-\int_0^1v^*\lambda-\int_0^1H(v,t)dt-\eta\int_0^1F(v,t)dt\;.
\eea
Critical points $(v,\eta)$ of $\A^\Mp$ are solutions of 
\beq\label{eqn:critical_points_eqn}
\left\{
\begin{aligned}
&\partial_tv=\eta X_{F}(v,t)+X_H(v,t),\;\forall t\in S^1\\
&\int_0^1F(v,t)dt=0
\end{aligned}\right.
\eeq
In \cite{Albers_Frauenfelder_Leafwise_intersections_and_RFH} we observed that critical points of the perturbed Rabinowitz action functional $\A^\Mp$ give rise to leaf-wise intersection points. 

\begin{Prop}[\cite{Albers_Frauenfelder_Leafwise_intersections_and_RFH}]\label{prop:critical_points_give_LI}
Let $(v,\eta)$ be a critical point of $\A^\Mp$ then $x:=v(\tfrac12)\in f^{-1}(0)$ and
\beq
\phi_H^1(x)\in L_x
\eeq
thus, $x$ is a leaf-wise intersection. Moreover, the map $\Crit\A^\Mp\to\{\text{leaf-wise intersections}\}$ is injective unless there exists a leaf-wise intersection point on a closed leaf, see figure \ref{fig:non-unique}.
\end{Prop}

\begin{figure}[htb]
\input{leafwise1.pstex_t}\caption{}\label{fig:non-unique}
\end{figure}

\subsubsection{First properties}\label{sec:first_properties}

In this paragraph we make the \underline{\textbf{wrong}} assumption that Rabinowitz Floer homology 
\beq
``\RFHb(\Sigma,M)\text{''}\equiv``\HM(\A^F)\text{''},
\eeq 
i.e.~the semi-infinite dimensional Morse homology (in the sense of Floer) for the Rabinowitz action functional, can always be constructed and is independent of the defining function $F$ for a regular hypersurface $\Sigma$ and invariant under Hamiltonian perturbations as described above. Thus, $\RFHb(\Sigma,M)$ should have the following properties:
\begin{enumerate}
\item $\RFHb(\Sigma,M)\cong\H_*(\Sigma)$ if there are no closed characteristics on $\Sigma$. Indeed, then the only critical points of $\A^F$ correspond to constant loops in $\Sigma$, see above. Then $\A^F$ is Morse-Bott with $\Crit\A^F\cong\Sigma$ and $\A^F|_{\Crit\A^F}=0$ since $\Sigma$ is a regular hypersurface.
\item If $\Sigma$ is displaceable, that is, $\exists H:M\times S^1\to\R$ such that $\phi_H^1(\Sigma)\cap\Sigma=\emptyset$, then $\RFHb(\Sigma,M)\cong0$. This follows from invariance under Hamiltonian perturbations
\beq
``\RFHb(\Sigma,M)\text{''}\equiv``\HM(\A^F)\text{''}\cong``\HM(\A^{(F,H)})\text{''}\cong0
\eeq
together with the observation that $\Crit\A^{(F,H)}=\emptyset$.
\end{enumerate}
The counterexamples to the Hamiltonian Seifert conjecture, see Ginzburg -- G\"urel \cite{Ginzburg_Gurel_A_C_2_smooth_counterexample_to_the_Hamiltonian_Seifert_conjecture} and the literature cited therein, are closed hypersurfaces $\Sigma\subset\R^{2n}$ with no closed characteristics. In particular, since any compact subset of $\R^{2n}$ is displaceable we arrive at the contradiction
\beq
\H_*(\Sigma)\cong``\RFHb(\Sigma,M)\text{''}\cong0\;.
\eeq 
The reason behind the fact that Rabinowitz Floer homology cannot be defined in full generality is that the moduli spaces of gradient flow lines do not have the necessary compactness properties. However, if certain topological/dynamical assumptions1 on $(\Sigma,M)$ are made the desired compactness properties can be established. This is described in the next section.

\subsection{Gradient flow lines}

\subsubsection{Gradient flow equation}

To compute the gradient of the Rabinowitz action functional we need to specify a metric on $\L\times\R$. We take the product of an $L^2$-metric on $\L$ and the standard metric on $\R$. In order to specify the $L^2$-metric on $\L$ we choose an $S^1$-family $J\equiv\{J_t\}_{t\in S^1}$ of $\om$-compatible almost complex structures on $M$. For $(\hat{v}_1,\hat{\eta}_1),(\hat{v}_2,\hat{\eta}_2)\in T_{(v,\eta)}\big(\L\times\R\big)=\Gamma(v^*TM)\times\R$ we set
\beq\label{eqn:L2_metric_RFH}
\mathfrak{m}\big((\hat{v}_1,\hat{\eta}_1),(\hat{v}_2,\hat{\eta}_2)\big)\equiv\mathfrak{m}_{J}\big((\hat{v}_1,\hat{\eta}_1),(\hat{v}_2,\hat{\eta}_2)\big):=\int_0^1\om\big(\hat{v}_1,J_t(v)\hat{v}_2\big)dt+\hat{\eta}_1\hat{\eta}_2\;.
\eeq
Then the gradient of $\A^F$ is
\beq
\nabla^\mathfrak{m}\A^F(v,\eta)=\left(
\begin{aligned}
J_t&(v)\big(\p_t v-\eta X_F(v)\big)\\
&-\int_0^1F(v)dt
\end{aligned}\right)\;.
\eeq
A gradient flow line is formally a map $w=(v,\eta)\in C^\infty(\R,\L\times \R)$ satisfying
\beq
\p_s w(s)+\nabla^\mathfrak{m}\A^F\big(w(s)\big)=0
\eeq
in the sense of Floer, that is $v:\R\times S^1\to M$, $\eta:\R\to\R$ satisfy
\beq
\left\{
\begin{aligned}
&\p_sv+J_t(v)\big(\p_t v-\eta X_F(v)\big)=0\\
&\p_s\eta-\int_0^1F(v)dt=0\;.
\end{aligned}\right.
\eeq
If we replace $\A^F$ by the perturbed Rabinowitz action functional $\A^{(F,H)}$ then the corresponding gradient flow equation changes to
\beq
\left\{
\begin{aligned}
&\p_sv+J_t(v)\big(\p_t v-\eta X_F(t,v)-X_H(t,v)\big)=0\\
&\p_s\eta-\int_0^1F(t,v)dt=0\;.
\end{aligned}\right.
\eeq

\subsubsection{Compactness}\label{sec:compactness}

\subsubsection*{\underline{The case of restricted contact type}}

Let $(W,\om=d\lambda)$ be a compact, exact symplectic manifold with contact type boundary $\Sigma=\p W$, that is, the Liouville vector field $L$, defined by $\iota_L\om=\lambda$, points outward along $\Sigma$. In particular, $(\Sigma,\alpha:=\lambda|_\Sigma)$ is a contact manifold.\footnote{
In general, a closed hypersurface in a symplectic manifold $(W,\om)$ is of \textit{restricted contact type} if there exists a globally defined primitive $\lambda$ of $\om$ with Liouville vector field $L$ satisfying $L\pitchfork \Sigma$. If $\lambda$ only exists locally near $\Sigma$ then $\Sigma$ is of \textit{contact type}.
}
 We denote by $M$ the completion of $W$ obtained by attaching the positive half of the symplectization of $\Sigma$, that is, $(M=W\cup_{\Sigma}(\Sigma\times\R_+), \om=d\lambda)$ where $\lambda$ is extended over $\Sigma\times\R_+ $ by $e^r\alpha$, $r\in\R_+$.  We assume that $F:M\to\R$ is a defining function for $\Sigma$ such that $dF$ has compact support. The main compactness theorem is as follows:
\begin{Thm}\label{thm:action_bounds_imply_compactness}
In the above situation let $w_n=(v_n,\eta_n)$ be a sequence of gradient flow lines of $\A^\Mp$ for which there exists $a<b$ such that
\beq
a\leq\A^\Mp\big(w_n(s)\big)\leq b\qquad\forall s\in\R\;.
\eeq
Then for every reparametrisation sequence $\sigma_n\in\R$ the sequence $w_n(\cdot+\sigma_n)$ has a subsequence which converges in $C^\infty_\mathrm{loc}(\R,\L\times\R)$.
\end{Thm}

The proof follows from standard arguments in Floer theory as soon as we establish
\begin{enumerate}
\item a uniform $C^0(\R)$ bound on $\eta_n$,
\item a uniform $C^0(\R\times S^1)$ bound on $v_n$,
\item a uniform $C^0(\R\times S^1)$ bound on the derivatives of $v_n$.
\end{enumerate} 
Assertion (2) follows from a maximum principle since $M$ is convex at infinity. Assertion (3) follows from standard bubbling-off analysis for holomorphic spheres in Floer theory together with the assumption that $(M,\om)$ is exact. Indeed, a non-constant holomorphic sphere in $M$ has to exist if the derivatives of $v_n$ explode . This contradicts Stokes theorem since $\om$ is exact. Obviously, symplectic asphericity of $M$ would be sufficient.

An interesting feature of gradient flow lines is that they have an infinite amount of flow time but a finite amount of energy available. This leads to the paradoxical conclusion that it is favorable for a gradient flow line to run slowly in order to get far away. The following ``Fundamental Lemma'' prevents that for slowly running gradient flow lines $w=(v,\eta)$ the Lagrange multiplier $\eta$ explodes, i.e.~slowly running gradient flow lines cannot get too far.

\begin{FunLemma}[restricted contact type case]
There exists a constant $C>0$ such that for all $(v,\eta)\in\L\times\R$
\beq\label{eqn:fund_lemma}
||\nabla^\mathfrak{m}\A^\Mp(v,\eta)||<\frac1C\quad\Longrightarrow\quad |\eta|\leq C(|\A^\Mp(v,\eta)|+1)
\eeq
\end{FunLemma}

\begin{Rmk}\label{rmk:Fundamental_Lemma_implies_uniform_bounds_on_eta}
With help of the Fundamental Lemma it can be proved that along gradient flow lines the Lagrange multiplier is bounded in terms of asymptotic action values. This depends heavily on the fact that the action $\A^\Mp$ is decreasing along gradient flow lines. For details we refer to Cieliebak -- Frauenfelder, \cite{Cieliebak_Frauenfelder_Restrictions_to_displaceable_exact_contact_embeddings}. 
\end{Rmk}

\begin{Rmk}
The assumption of restricted contact type enters the proof of the Fundamental Lemma through the following observation. If we normalize $F$ such that $X_F|_\Sigma$ is the Reeb vector field of $\alpha$ then in the unperturbed case at a critical points $(v,\eta)\in\Crit\A^F$ we have
\beq\label{eqn:critical_value_equal_eta}
\A^F(v,\eta)=-\eta\;.
\eeq 
Since $v(t/\eta)$ is a $\eta$-periodic Reeb orbit equation \eqref{eqn:critical_value_equal_eta} can be thought of as a period-action equality for Reeb orbits. Consequently, equation \eqref{eqn:fund_lemma} can be interpreted as a period-action inequality for almost Reeb orbits.
\end{Rmk}

\begin{Rmk}
For the unperturbed Rabinowitz action functional the Fundamental Lemma and Theorem \ref{thm:action_bounds_imply_compactness} have been proved in \cite{Cieliebak_Frauenfelder_Restrictions_to_displaceable_exact_contact_embeddings} and for the perturbed Rabinowitz action functional in \cite{Albers_Frauenfelder_Leafwise_intersections_and_RFH}.
\end{Rmk}

\begin{Rmk}
If the defining function $F$ and the almost complex structure $J$ are adapted to the symplectization structure near the hypersurface $\Sigma$ then the constant $C$ in the Fundamental Lemma can be chosen universally, i.e.~independent of $F$, $J$, and the contact structure. This has been used in the definition of spectral invariants for Rabinowitz Floer homology in \cite{Albers_Frauenfelder_Spectral_invariants_in_RFH}.
\end{Rmk}

\begin{Rmk}
The assumption that the differential of the defining function $F:M\to\R$ has compact support can be replaced by an appropriate asymptotic growth condition without changing the Rabinowitz Floer homology. This property plays an important role in the proof relating Rabinowitz Floer homology to symplectic homology and loop space topology, see \cite{Cieliebak_Frauenfelder_Oancea_Rabinowitz_Floer_homology_and_symplectic_homology,
Abbondandolo_Schwarz_Estimates_and_computations_in_Rabinowitz_Floer_homology,
Merry_On_the_RFH_of_twisted_cotangent_bundles}.
\end{Rmk}

\begin{Rmk}
Let $(M,\om)$ be an exact symplectic manifold which is convex at infinity. If $\Sigma\subset M$ is a compact, bounding hypersurface of restricted contact type then the symplectization of $\Sigma$ embeds into $M$. It has been examined by Cieliebak -- Frauenfelder -- Oancea in \cite{Cieliebak_Frauenfelder_Oancea_Rabinowitz_Floer_homology_and_symplectic_homology} under which conditions gradient flow lines of the Rabinowitz action functional do not leave the positive part of the symplectization of $\Sigma$. The upshot is that $M$ can be replaced by the completion of the region bounded by $\Sigma$ without changing the Rabinowitz Floer homology.
\end{Rmk}

\subsubsection*{\underline{The case of stable hypersurfaces}}

Let $(M,\om)$ be a symplectically aspherical symplectic manifold. We assume further that $(M,\om)$ is convex at infinity or geometrically bounded to guarantee $C^0$-bounds for gradient flow lines. A closed hypersurface  $\Sigma\subset M$ is called stable if there exists $\alpha\in\Om^1(\Sigma)$ such that
\beq
\alpha\wedge\om^{m-1}|_{\Sigma}>0\quad\text{and}\quad \ker\om\subset\ker d\alpha 
\eeq
where $m=\tfrac12\dim M$. The 1-form $\alpha$ is called a stabilizing 1-form. $\Sigma$ is called of contact type if $\om|_\Sigma=d\alpha$. Since $\ker\om|_\Sigma\to\Sigma$ is a rank-1-distribution the Reeb vector field $R$ on $\Sigma$ is uniquely defined by
\beq
\iota_R\alpha=1 \quad\text{and}\quad \iota_R\om|_\Sigma=0\;.
\eeq
The Rabinowitz action functional $\A^F$ is defined as above. However, the Fundamental Lemma does not carry over to the stable case since already at critical points the period-action equality fails to hold for $\A^F$ in general. We define a modified Rabinowitz action functional by
\bea
\widehat{\A}^F:\L\times\R&\pf\R\\
(v,\eta)&\mapsto \widehat{\A}^F(v,\eta):=-\int_{S^1}v^*\lambda-\eta\int_0^1F\big(v(t)\big)dt\;.
\eea
where $\lambda\in\Om^1(M)$ satisfies $\lambda|_\Sigma=\alpha$. Then the period-action equality holds for $\widehat{\A}^F$. If $\lambda$ is chosen appropriately then the Fundamental Lemma in the stable case is as follows.

\begin{FunLemma}[stable case]
There exists a constant $C>0$ such that for all $(v,\eta)\in\L\times\R$
\beq
||\nabla^\mathfrak{m}\A^F(v,\eta)||<\frac1C\quad\Longrightarrow\quad |\eta|\leq C(|\widehat{\A}^F(v,\eta)|+1)
\eeq
\end{FunLemma}

\noindent This was proved by Cieliebak -- Frauenfelder -- Paternain in \cite{Cieliebak_Frauenfelder_Paternain_Symplectic_Topology}. Establishing uniform bounds on the Lagrange multiplier along gradient flow lines in terms of asymptotic action values as described in Remark \ref{rmk:Fundamental_Lemma_implies_uniform_bounds_on_eta} in the stable case involves additional arguments. This is due to the fact that $\widehat{\A}^F$ is not necessarily decreasing along gradient flow lines of $\A^F$. The crucial observation is that 
\beq\label{eqn:aux_fctl_is_non_increasing}
\A_\alpha(v,\eta):=\A^F(v,\eta)-\widehat{\A}^F(v,\eta)=\int v^*\lambda-\int \bar{v}^*\om:\L_0\times\R\pf\R
\eeq
is non-increasing along gradient flow lines of $\A^F$, i.e.~$\A_\alpha$ serves as a Lyapunov function for the gradient flow of $\A^F$. This immediately implies that if $w$ is a gradient flow line of $\A^F$ with $\lim_{s\to\pm\infty}w(s)=:w_\pm\in\Crit\A^F$ we can estimate
\beq
\A^F(w_+)-\A_\alpha(w_-)\leq\widehat{\A}^F(w(s))\leq\A^F(w_-)-\A_\alpha(w_+)\;.
\eeq
Proving that $\A_\alpha$ is non-increasing along gradient flow lines involves the bilinear form
\beq
\widehat{\mathfrak{m}}\big((\hat{v}_1,\hat{\eta}_1),(\hat{v}_2,\hat{\eta}_2)\big):=\int_0^1d\lambda\big(\hat{v}_1,J_t(v)\hat{v}_2\big)dt+\hat{\eta}_1\hat{\eta}_2\;.
\eeq
 for $(\hat{v}_1,\hat{\eta}_1),(\hat{v}_2,\hat{\eta}_2)\in T_{(v,\eta)}\big(\L\times\R\big)=\Gamma(v^*TM)\times\R$. We note that in general $\widehat{\mathfrak{m}}$ is neither symmetric nor positive definite. 
\begin{Prop}\label{prop:properties_of_m_hat}
If $F$, $J$, and $\lambda$ are chosen appropriately, see \cite{Cieliebak_Frauenfelder_Paternain_Symplectic_Topology}, the following two properties hold
\bea
&(1) \quad\mathfrak{m}-\widehat{\mathfrak{m}}\geq0\\
&(2) \quad d\widehat{\A}^F(\,\cdot\,)=\widehat{\mathfrak{m}}(\nabla^{\mathfrak{m}}\A^F,\cdot)\;.
\eea
\end{Prop}

\begin{Rmk}
Property (2) in Proposition \ref{prop:properties_of_m_hat} can be interpreted as $\nabla^{\widehat{\mathfrak{m}}}\A^{\widehat{F}}=\nabla^{\mathfrak{m}}\A^F$ although $\nabla^{\widehat{\mathfrak{m}}}\A^{\widehat{F}}$ is not well-defined since $\widehat{\mathfrak{m}}$ might be degenerate.
\end{Rmk}

\begin{Rmk}
Proposition \ref{prop:properties_of_m_hat} is very sensitive to Hamiltonian perturbations. It is not clear if compactness continues to hold in the perturbed case. However, Kang proved  in \cite{Kang_Existence_of_leafwise_intersection_points_in_the_unrestricted_case} an analogue of Proposition \ref{prop:properties_of_m_hat} in the case that the hypersurface is of contact type (but not necessarily restricted contact type) and that the Hamiltonian perturbation is localized in the symplectization.
\end{Rmk}

\begin{Cor}
If $F$, $J$, and $\lambda$ are chosen as in Proposition \ref{prop:properties_of_m_hat} then $\A_\alpha$ is a Lyapunov function for the gradient flow of $\A^F$. 
\end{Cor}

\begin{proof}
Let $w$ be a gradient flow line of $\A^F$, i.e.~a solution of  $\p_s w(s)+\nabla^\mathfrak{m}\A^F\big(w(s)\big)=0$. Then we estimate using Proposition \ref{prop:properties_of_m_hat}
\bea
\frac{d}{ds}\A_\alpha(w(s))&=d\A^F(w(s))[\p_s w]-d\widehat{\A}^F(w(s))[\p_s w]\\
&=d\widehat{\A}^F(w(s))[\nabla^\mathfrak{m}\A^F]-d\A^F(w(s))[\nabla^\mathfrak{m}\A^F]\\
&=\widehat{\mathfrak{m}}(\nabla^{\mathfrak{m}}\A^F, \nabla^{\mathfrak{m}}\A^F)-{\mathfrak{m}}(\nabla^{\mathfrak{m}}\A^F, \nabla^{\mathfrak{m}}\A^F)\\
&\leq0\;.
\eea
\end{proof}

\subsubsection{Relation to symplectic vortex equations}

Symplectic vortex equations have been found independently by Mundet and Salamon and were known in the physics literature as gauged sigma models. They are usually studied for Hamiltonian group actions of compact groups. Symplectic vortex equation with domain the 2-dimensional cylinder can be derived from the gradient flow equation of the moment action functional. We explain in this section how the gradient flow equation for the Rabinowitz action functional can be thought of as the substitute for the (not well-defined) symplectic vortex equations for Hamiltonian $\R$-actions. 

Let $(M,\om)$ be a symplectically aspherical  manifold and $F:M\to\R$ a compactly supported autonomous Hamiltonian function. Then the Hamiltonian flow $\phi_F^t:M\to M$ can be thought of as a Hamiltonian $\R$-action with moment map $F$. The moment action functional corresponding to this $\R$-action as introduced by Cieliebak -- Gaio -- Salamon in \cite{Cieliebak_Gaio_Salamon_J_hol_curves_moment_maps_and_invariants_of_Hamiltonian_group_actions} is
\bea
\mathbb{A}^F:\L_0\times C^\infty(S^1,\R)&\pf\R\\
(v,\eta)&\mapsto\mathbb{A}(v,\eta):=-\int \bar{v}^*\om-\int\eta(t)F(v(t)) dt\;.
\eea
This differs from the Rabinowitz action functional in so far that we have an entire loop $\eta(t)$ of Lagrange multiplies, i.e.~there are infinitely many Lagrange multipliers. Critical points of $\mathbb{A}^F$ are solutions of
\beq
\left\{
\begin{aligned}
&\partial_tv=\eta(t) X_{F}(v),\;\forall t\in S^1\\
&F(v(t))=0,\;\forall t\in S^1\;.
\end{aligned}\right.
\eeq
The gauge group $\mathscr{H}:=C^\infty(S^1,\R)$ acts on $\L_0\times C^\infty(S^1,\R)$ by 
\beq
(v,\eta)\mapsto h_*(v,\eta):=(\phi_F^{h(t)}(v(t)),\eta+h')\;. 
\eeq
The differential of $\mathbb{A}^F$ is invariant under the action of the gauge group:
\beq
d\mathbb{A}^F(w)\big[\hat{w}\big]=d\mathbb{A}^F(h_*w)\Big[dh(w)\big[\hat{w}\big]\Big]
\eeq
for all $w\in\L_0\times C^\infty(S^1,\R)$, $\hat{w}\in T_w\big(\L_0\times C^\infty(S^1,\R)\big)$, and $h\in\mathscr{H}$. Since $\mathscr{H}$ is connected we conclude that $\mathbb{A}^F$ is invariant under $\mathscr{H}$, too. We define the based gauge group
\beq
\mathscr{H}_0:=\{h\in\mathscr{H}\mid h(0)=0\}\;.
\eeq
The based gauge group acts freely on $\L_0\times C^\infty(S^1,\R)$. Since for any $\eta\in C^\infty(S^1,\R)$ there exists a unique $h_\eta\in\mathscr{H}_0$ such that $\eta+h_\eta'=\int_0^1\eta(t)dt$ we can identify 
\bea\label{eqn:Coulomb_map}
\mathscr{C}:\L_0\times C^\infty(S^1,\R)\big/\mathscr{H}_0&\cong\L_0\times\R\\
[(v,\eta)]&\mapsto (h_\eta)_*(v,\eta)=\left(\phi_F^{h_\eta(t)}\big(v(t)\big),\int_0^1\eta(t)dt\right)\;.
\eea
Taking the mean-value of $\eta$ corresponds to the Coulomb gauge on the trivial $\R$-bundle over $S^1$. 

\begin{Rmk}
It is a remarkable fact that there exists a global slice for the gauge group action. This is related to the fact that $\R$ is abelian. For non-abelian gauge theories in general there exist no global slices. If the moment action functional $\mathbb{A}^F$ is restricted to this global Coulomb slice we obtain the Rabinowitz action functional $\A^F$. This gives another explanation why a single Lagrange multiplier in the Rabinowitz action functional $\A^F$ eventually leads to a point-wise constraint at critical points (as opposed to an integral constraint).
\end{Rmk}

In order to proceed to the symplectic vortex equations we need to assume that there exists an $\R$-invariant compatible almost complex structure $J$ on $M$, i.e.
\beq
(\phi_F^t)_*J=J\quad\forall t\in\R\;.
\eeq
On compact Lie groups invariant almost complex structure exist by averaging. In the current situation existence is not at all clear but occurs for instance if the flow of $F$ is periodic. Given such an invariant $J$ we define an $L^2$ inner product on $\L_0\times C^\infty(S^1,\R)$ by
\beq
{\mathfrak{g}}\big((\hat{v}_1,\hat{\eta}_1),(\hat{v}_2,\hat{\eta}_2)\big):=\int_0^1\om\big(\hat{v}_1,J(v)\hat{v}_2\big)dt+\int_0^1\hat{\eta}_1(t)\hat{\eta}_2(t)dt
\eeq
for $(\hat{v}_i,\hat{\eta}_i)\in T_{(v,\eta)}(\L_0\times C^\infty(S^1,\R))$. Then the gradient of $\mathbb{A}^F$ equals
\beq
\nabla^\mathfrak{g}\mathbb{A}^F(v,\eta)=\left(
\begin{aligned}
J(v)\big(&\p_t v-\eta X_F(v)\big)\\
&F(v)dt
\end{aligned}\right)
\eeq
and the gradient flow equation is
\beq\tag{$\star$}
\left\{
\begin{aligned}
&\p_sv+J(v)\big(\p _tv-\eta X_F(v)\big)=0\\
&\p_s\eta-F(v)=0
\end{aligned}\right.
\eeq
where $v:\R\times S^1\to M$ and $\eta:\R\times S^1\to\R$ are smooth maps. Since the functional $\mathbb{A}$ and the metric $\mathfrak{g}$ are invariant under $\mathscr{H}=C^\infty(S^1,\R)$ so is the gradient flow equation. These are the symplectic vortex equations on the cylinder in temporal gauge. The full symplectic vortex equations are
\beq\tag{$\star\star$}
\left\{
\begin{aligned}
&\p_sv-\zeta X_F(v)+J(v)\big(\p _tv-\eta X_F(v)\big)=0\\
&\p_s\eta-\p_t\zeta=F(v)
\end{aligned}\right.
\eeq
where $v:\R\times S^1\to M$ and $\eta,\zeta:\R\times S^1\to\R$ are smooth maps. This is invariant under the enlarged gauge group $\mathscr{G}:=C^\infty(\R\times S^1,\R)$ acting by
\beq
g_*(v,\eta,\zeta)=(\phi_F^g(v),\eta+\p_tg,\zeta+\p_sg)\;.
\eeq
Then we have
\beq
\{(v,\eta)\text{ solves }(\star)\}\big/\mathscr{H}\cong\{(v,\eta,\zeta)\text{ solves }(\star\star)\}\big/\mathscr{G}
\eeq
since in each $\mathscr{G}$-orbit $[v,\eta,\zeta]$ there is a representative with $\zeta=0$ obtained by regauging with $g\in\mathscr{G}$ satisfying $\p_sg=-\zeta$.
The second equation in $(\star\star)$ should be thought of as an equation for the curvature of the connection $\mathbf{A}:=\eta dt+\zeta ds$ on the cylinder. Indeed $(\star\star)$ can be written as
\beq\tag{$\star\star'$}
\left\{
\begin{aligned}
&\p_{J,\mathbf{A}}v=0\\
&\ast F_\mathbf{A}(v)=F(v)
\end{aligned}\right.
\eeq
where $F_\mathbf{A}$ is the curvature of $\mathbf{A}$ and $\ast$ is the Hodge star operator.

\begin{Rmk}
We point out that we have two inner products on $ \L_0\times\R$, see diagram \eqref{diagram}. One is the natural $L^2$ inner product $\mathfrak{m}$, see \eqref{eqn:L2_metric_RFH}. The other is $\mathscr{C}^*[\mathfrak{g}]$ where $\mathscr{C}$ is the Coulomb map, see \eqref{eqn:Coulomb_map}. We recall that $\mathfrak{g}$ is $\mathscr{H}_0$-invariant and we denote by $[\mathfrak{g}]$ the induced metric on the quotient.
\begin{equation}\label{diagram}
\begin{xy}
\xymatrix{
 &\big(\L_0\times C^\infty(S^1,\R),\mathfrak{g}\big)\ar[d]^{\pi}\\
 \big(\L_0\times\R,\mathfrak{m}\big) \ar@{^{(}->}[ru]^{\iota}&\ar[l]_>{\quad\qquad\mathscr{C}} \big(\L_0\times C^\infty(S^1,\R)\big/\mathscr{H}_0,[\mathfrak{g}]\big)
}
\end{xy}
\end{equation}
We point out that
\beq
\mathfrak{m}\neq\mathscr{C}^*[\mathfrak{g}]
\eeq
which is due to the fact that the infinitesimal gauge action is not $\mathfrak{g}$-orthogonal to the Coulomb slice. Thus, gradient  flow lines of $\nabla^\mathfrak{m}\A^F$ are different from gauge orbits of gradient flow lines of $\nabla^{\mathscr{C}^*[\mathfrak{g}]}\A^F$. The latter are the symplectic vortex equations. Therefore, the gradient flow equation for the Rabinowitz action functional is substantially different from the symplectic vortex equations. This also reflected in the respective compactness proofs for gradient flow lines.
\end{Rmk}

\subsubsection{Definition of Rabinowitz Floer homology}

For a generic Moser pair $\Mp=(F,H)$ the perturbed Rabinowitz action functional $\A^\Mp$ is Morse. Interestingly enough, if $\Sigma$ is a regular hypersurface, it is enough to perturb $H$ in order to make $\A^\Mp$ Morse, see \cite{Albers_Frauenfelder_Leafwise_intersections_and_RFH}. On the other hand, if $H=0$ then $\A^F$ is never Morse. This is due to the fact that $\A^F$ is invariant under the $S^1$-action on the loop space $\L$. Consequently, $S^1$ acts on the critical set $\Crit\A^F$. Indeed,  if $(v,\eta)\in\Crit\A^F$ is a Reeb orbit then $v(t+\tau,\eta)$, $\tau\in S^1$, is also a critical point. Moreover, $\Crit\A^F$ always contains the constant solutions $(x,0)$, $x\in\Sigma$. These are the fixed points of the $S^1$-action. For generic choice of $F$ the unperturbed Rabinowitz action functional $\A^F$ is Morse-Bott, see \cite{Cieliebak_Frauenfelder_Restrictions_to_displaceable_exact_contact_embeddings}.

In the following we only give a definition of Rabinowitz Floer homology in the case where $\A^\Mp$ is Morse. The Morse-Bott case can be defined by choosing an auxiliary Morse function on the critical manifolds and  counting gradient flow lines with cascades, see Frauenfelder \cite{Frauenfelder_Arnold_Givental_Conjecture,Cieliebak_Frauenfelder_Restrictions_to_displaceable_exact_contact_embeddings}. Moreover, we restrict our attention to $\Z/2$-coefficients. There's no doubt that Rabinowitz Floer homology can also be defined with $\Z$-coefficients but so far there is no treatment of coherent orientations for Rabinowitz Floer homology in the literature.
\subsubsection*{\underline{The case of restricted contact type}}
Let $\Sigma$ be a hypersurface of restricted contact type and $\Mp$ a Moser pair such that $\A^\Mp$ is Morse. The Rabinowitz Floer chain complex is defined as the Morse complex of $\A^\Mp$ with help of Novikov sums, see Hofer -- Salamon, \cite{Hofer_Salamon_Floer_homology_and_Novikov_rings}, for Floer theory with Novikov rings:
\beq
\RFHb(\Mp):=\Big\{\xi=\sum_{c}\xi_c\,c\mid\#\{c\in\Crit{\A^\Mp}\mid\xi_c\neq0\in\Z/2\text{ and }\A^\Mp(c)\geq\kappa\}<\infty\;\forall\kappa\in\R\Big\}\;.
\eeq
The boundary operator 
\beq
\p:\RFCb(\Mp)\pf\RFCb(\Mp)
\eeq
is defined by counting gradient flow lines by the standard procedure in Floer homology and satisfies $\p\circ\p=0$ thanks to compactness. Then Rabinowitz Floer homology is
\beq
\RFHb(\Mp):=\H(\RFCb(\Mp),\p))\;.
\eeq
The usual theory of continuation homomorphisms in Floer theory implies that Rabinowitz Floer homology is independent of auxiliary data such as the almost complex structure $J$ and the perturbation $H$. Moreover, if the hypersurface $\Sigma$ is isotoped through restricted contact type hypersurfaces Rabinowitz Floer homology does not change either.

\begin{Rmk}
Since Rabinowitz action functional is defined on the full loop space and the differential in the Rabinowitz Floer complex counts topological cylinders we can split Rabinowitz Floer homology into factors labeled by free homotopy classes
\beq
\RFHb(\Mp)=\bigoplus_{\gamma\in[S^1,M]}\RFH(\Mp,\gamma)\;.
\eeq
\end{Rmk}

\begin{Def}
We abbreviate by
\beq
\RFH(\Sigma,M):=\RFH(\Mp,\mathrm{pt})
\eeq 
where $\mathrm{pt}$ denotes the free homotopy class of a point and
\beq
\mathbf{RFH}(\Sigma,M):=\RFHb(\Mp)
\eeq
where $\Mp$ is a Moser pair such that $\A^\Mp$ is Morse. 
\end{Def}

\begin{Def}
Let $(W,\om=d\lambda)$ be an exact, compact symplectic manifold with contact type boundary $\Sigma:=\p W$. We denote by $\widehat{W}$ the completion of $W$ as described at the beginning of section \ref{sec:compactness} and set
\bea
\RFHb(\p W, W)&:=\RFHb(\Sigma,\widehat{W})\;,\\
\RFH(\p W, W)&:=\RFH(\Sigma,\widehat{W})\;.
\eea
\end{Def}

\begin{Thm}[\cite{Cieliebak_Frauenfelder_Oancea_Rabinowitz_Floer_homology_and_symplectic_homology}]\label{thm:RFH_does_not_depend_on_exterior}
Let $\Sigma\subset(M,\om=d\lambda)$ be a restricted contact type closed hypersurface which bounds the compact region $W$, then Rabinowitz Floer homology does not depend on the exterior $M\setminus W$:
\bea
\RFHb(\Sigma, W)&=\RFHb(\Sigma,M)\;,\\
\RFH(\Sigma,W)&=\RFH(\Sigma,M)\;.
\eea
\end{Thm}

There are two other versions of Rabinowitz Floer homology. For this we fix numbers $a<b$ and define the $\Z/2$ vector space
\beq
\RFCb^{(a,b)}(\Mp):=\Big\{\sum_{c}\xi_c c\mid c\in\Crit\Mp, a<\A^\Mp(c)<b, \xi_c\in\Z/2\Big\}\;.
\eeq
Counting gradient flow lines again defines a differential on $\RFCb^{(a,b)}(\Mp)$. We denote the homology by
\beq
\RFHb^{(a,b)}(\Mp)\;.
\eeq 
The natural inclusion and projection homomorphisms between $\RFCb^{(a,b)}(\Mp)$ and $\RFCb^{(a',b')}(\Mp)$ turn $\Big\{\RFHb^{(a,b)}(\Mp)\Big\}$ into a bidirected system of vector spaces. Thus, we can take direct and inverse limits:
\bea
\overline{\RFHb}(\Mp)&:=\lim_{\substack{\longrightarrow\\b\to\infty}}\lim_{\substack{\longleftarrow\\a\to-\infty}}\RFHb^{(a,b)}(\Mp)\;,\\
\underline{\RFHb}(\Mp)&:=\lim_{\substack{\longleftarrow\\a\to-\infty}}\lim_{\substack{\longrightarrow\\b\to\infty}}\RFHb^{(a,b)}(\Mp)\;.
\eea
The three vector spaces $\RFHb(\Mp)$, $\overline{\RFHb}(\Mp)$, and $\underline{\RFHb}(\Mp)$ are related by a commutative diagram as follows
\begin{equation}\label{dia}
\begin{xy}
 \xymatrix{\RFHb(\Mp) \ar[r]^{\overline{\rho}}
\ar[rd]_/-0.8em/{\underline{\rho}}& \overline{\RFHb}(\Mp)\ar[d]^{\kappa}\\
 & \underline{\RFHb}(\Mp)
}
\end{xy}
\end{equation}
in which $\overline{\rho}$ is an isomorphism and $\kappa$ is surjective. We point out that the canonical map $\overline{\rho}$ is not necessarily an isomorphism if $\Z$-coefficients instead of field coefficients are used. For details and proofs see \cite{Cieliebak_Frauenfelder_Morse_homology_on_noncompact_manifolds}.

\begin{Rmk}\label{rmk:virtual_restricted_contact_type}
A hypersurface $\Sigma\subset(M,\om)$ is of virtual restricted contact type if there exists a cover $\pi:\widetilde{M}\pf M$ and a primitive $\widetilde{\lambda}$ of $\pi^*\om$ which is bounded on $\pi^{-1}(\Sigma)$ with respect to $\pi^*g$ where $g$ is a Riemannian metric on $\Sigma$. The notion of virtual restricted contact type plays an important role in the study of twisted cotangent bundles since energy hypersurfaces above Ma\~n\'e's critical value are of virtual restricted contact type but in general not of restricted contact type, see \cite{Cieliebak_Frauenfelder_Paternain_Symplectic_Topology} for details. Rabinowitz Floer homology for virtual restricted contact type hypersurfaces can be defined as in the case of restricted contact type on the component of the loop space containing the contractible loops.
\end{Rmk}

\subsubsection*{\underline{The case of stable hypersurfaces}}
In this section the Rabinowitz action functional is only considered on the space of contractible loops. If restricted contact type is replaced by stability of the hypersurface two difficulties appear. The first is of technical nature namely stability is not an open condition, see \cite{Cieliebak_Frauenfelder_Paternain_Stability_is_not_open}. Therefore, in general it cannot be expected that the unperturbed Rabinowitz action functional can be made Morse-Bott on a given stable hypersurface. It remains true that for a small Hamiltonian perturbation the perturbed Rabinowitz action function is Morse. However, Proposition \ref{prop:properties_of_m_hat} does not allow for Hamiltonian perturbations. Therefore, compactness for families of gradient flow lines might fail as follows: new gradient flow lines from infinity, i.e.~with very large Lagrange multipliers, might appear. Nevertheless, for sufficiently small perturbation there is no interaction between gradient flow lines with small and large Lagrange multipliers. To implement this the filtration on the Rabinowitz Floer complex has to be  modified to involve both action functionals $\A^F$ and $\A_\alpha$ defined above. Since both action functional are Lyapunov functions along gradient flow lines (see above) we obtain a doubly filtered complex $\RFC^{(a,b),(a',b')}$. Then using a Hamiltonian perturbation of size depending on $a,b,a',b'$ a boundary operator $\p$ can still be defined by ignoring gradient flow lines with very large Lagrange multipliers. Moreover, the resulting filtered homology is independent of the small Hamiltonian perturbation. Thus, $\overline{\RFH}$ can be defined as above by taking an inverse-direct limit over $a,b,a',b'$. We point out that there is no analogue of the definition with Novikov sums for stable hypersurfaces.

The second, more serious difficulty is that the just defined Rabinowitz Floer homology might depend on the stabilizing 1-form $\alpha\in\Om^1(\Sigma)$ in a very subtle way. Indeed, even though the critical points and the gradient flows do not depend on $\alpha$ Rabinowitz Floer homology $\overline{\RFH}$ might depend on $\alpha$ through the double action filtration by $\A^F$ and $\A_\alpha$ in the inverse-direct limit. To guarantee independence of Rabinowitz Floer homology we impose the additional assumption of tameness. 

\begin{Def}
A stable hypersurface $\Sigma\in(M,\om)$ is tame if there exists a stabilizing 1-form $\alpha\in\Om^1(\Sigma)$ and a constant $C>0$ such that
\beq
\left|\int v^*\alpha\right|\leq C\left|\int\bar{v}^*\om\right|
\eeq
for all Reeb orbits $v$ of $\alpha$ which are contractible in $M$. Here $\bar{v}$ again denotes a filling disk and the right hand side is independent of $\bar{v}$ since $(M,\om)$ is symplectically aspherical.
\end{Def}
There are examples of stable hypersurfaces which are not tame. In fact, even contact type hypersurface need not be tame. For more details we refer to \cite{Cieliebak_Frauenfelder_Paternain_Symplectic_Topology}. More examples of non-tame hypersurfaces are provided in Macarini -- Paternain \cite{Macarini_Paternain_On_the_stability_of_Mane_critical_hypersurfaces} and Cieliebak -- Volkov \cite{Cieliebak_Volkov_First_steps_in_stable_Hamiltonian_topology}. An isotopy of hypersurfaces is called a stable isotopy if the stabilizing 1-form depends smoothly on the isotopy parameter. It is called a tame stable isotopy if in addition all hypersurfaces are tame with a taming constant independent of the isotopy parameter.

\begin{Thm}[\cite{Cieliebak_Frauenfelder_Paternain_Symplectic_Topology}]
If $\Sigma$ is tame then Rabinowitz Floer homology $\overline{\RFH}$ is invariant under tame stable isotopies. Since the space of stabilizing 1-forms is a cone, hence connected, $\overline{\RFH}$ is independent of $\alpha$.
\end{Thm}

\begin{Def}
We abbreviate
\beq
\RFH(\Sigma,M):=\overline{\RFH}\;.
\eeq 
This is justified since in the case of restricted contact type the natural map $\overline{\rho}$ in the commuting diagram \ref{dia} is an isomorphism.
\end{Def}

\subsubsection*{\underline{Grading}} 

If the homomorphism $I_{c_1(M)}:\pi_2(M)\to\Z$ obtained by integrating the first Chern class $c_1(M)$ over a smooth representative vanishes identically then Rabinowitz Floer homology $\RFH_*$ carries a $\Z$-grading in terms of the Conley-Zehnder index, see \cite{Cieliebak_Frauenfelder_Restrictions_to_displaceable_exact_contact_embeddings}. In general only  a relative $\Z/N$ can be defined where $N$ is the minimal Chern number. There are different normalization conventions. In \cite{Cieliebak_Frauenfelder_Restrictions_to_displaceable_exact_contact_embeddings} the grading takes values in the set $\frac12+\Z$. Underlying this convention is the philosophy that the index in a semi-infinite homology is $\tfrac12$ times some kind of signature index of the Hessian at critical points. We refer to the article  \cite{Robbin_Salamon_Feynman_path_integrals} by Salamon -- Robbin for research in this direction.  The  half signature index takes values in $\Z$ on even dimensional manifolds and in $\frac12+\Z$ on odd dimensional manifolds. Since the loop space $\L$ of a symplectic manifolds is itself symplectic $\L$ is an ``even'' infinite-dimensional manifold and consequently $\L\times\R$ is an ``odd'' infinite-dimensional manifold. Therefore the convention used in \cite{Cieliebak_Frauenfelder_Restrictions_to_displaceable_exact_contact_embeddings} is consistent with this interpretation. Moreover, with this convention 
\beq
\RFH^{-*}\cong\RFH_*
\eeq
holds. This isomorphism is induced by the involution
\bea
I:\L\times\R&\pf\L\times\R\\
(v,\eta)&\mapsto (v^-,-\eta)
\eea
where $v^-(t):=v(1-t)$ under which the unperturbed Rabinowitz action functional is anti-invariant
\beq
\A^F\circ I=-\A^F\;.
\eeq
Another useful convention is to replace the above convention by adding $\frac12$. This conventions fits better with the computational results comparing Rabinowitz Floer homology with symplectic homology and loop space topology as carried out in \cite{Cieliebak_Frauenfelder_Oancea_Rabinowitz_Floer_homology_and_symplectic_homology,
Abbondandolo_Schwarz_Estimates_and_computations_in_Rabinowitz_Floer_homology,
Merry_On_the_RFH_of_twisted_cotangent_bundles}.

\subsection{Computations}

\subsubsection{The displaceable case}

Let $(M,\om)$ be symplectically aspherical, and convex at infinity or geometrically bounded and $\Sigma\subset(M,\om)$ of (virtual) restricted contact type or stable tame. 

\begin{Thm}$ $
\begin{enumerate}
\item If the hypersurface $\Sigma$ is displaceable then Rabinowitz Floer homology $\RFH_*(\Sigma,M)$ vanishes.
\item If in addition $\Sigma$ is of restricted contact type then even the full Rabinowitz Floer homology $\RFHb_*(\Sigma,M)$ vanishes.
\end{enumerate}
\end{Thm}

This is proved for restricted contact type in \cite{Cieliebak_Frauenfelder_Restrictions_to_displaceable_exact_contact_embeddings} and for the remaining cases in \cite{Cieliebak_Frauenfelder_Paternain_Symplectic_Topology}.

\begin{Thm}
If the hypersurface $\Sigma$ carries no Reeb orbits which are contractible in $M$ then
\beq
\RFH_*(\Sigma,M)\cong\H_*(\Sigma)\;.
\eeq 
\end{Thm}

This follows immediately from the definition of $\RFH_*$ in the Morse-Bott case, see above. The previous two Theorems imply the following Corollary.

\begin{Cor} \label{cor:existence_of_Reeb_orbit_on_displaceable_stable}
If the hypersurface $\Sigma$ is displaceable then there exists a Reeb orbit which is contractible in $M$.
\end{Cor}

\begin{Rmk}
If the hypersurface $\Sigma$ is stable then Corollary  \ref{cor:existence_of_Reeb_orbit_on_displaceable_stable} has been proved by Schlenk in \cite{Schlenk_Applications_of_Hofers_geometry_to_Hamiltonian_dynamics} using different methods. Schlenk does not need the tameness assumption. Using a local version of Rabinowitz Floer homology Corollary \ref{cor:existence_of_Reeb_orbit_on_displaceable_stable} can also be proved without the tameness assumption in the framework of Rabinowitz Floer homology, see \cite{Cieliebak_Frauenfelder_Paternain_Symplectic_Topology}.
\end{Rmk}

\subsubsection{Relations to symplectic homology and loop space topology}

For the next theorem we assume that we are in the same setup as at the beginning of section \ref{sec:compactness}, \textit{The case of restricted contact type}. Namely, $(W,\om=d\lambda)$ is a compact exact symplectic manifold with boundary $\p W=\Sigma$ of restricted contact type and $(M,\om)$ is the completion of $W$. In this situation symplectic homology $\SH_*(\Sigma,M)$ resp.~cohomology $\SH^*(\Sigma,M)$  of $M$ can be defined, see Cieliebak -- Floer -- Hofer \cite{Cieliebak_Floer_Hofer_Symplectic_homology_II_A_general_construction} and Viterbo \cite{Viterbo_partI}.

\begin{Thm}[\cite{Cieliebak_Frauenfelder_Oancea_Rabinowitz_Floer_homology_and_symplectic_homology}]
There exists a long exact sequence
\beq
\cdots\to\SH^{-*}(\Sigma,M)\to\SH_*(\Sigma,M)\to\RFHb_*(\Sigma,M)\to\SH^{-*+1}(\Sigma,M)\to\cdots
\eeq
\end{Thm}
From now on $M=T^*B$ is the cotangent bundle with its standard symplectic structure $\om_{std}=d\lambda_{std}$ over a closed manifold $B$ and $\Sigma=S^*_gB$ is the unit cotangent bundle with respect to some Riemannian metric $g$ on $B$. Using that the symplectic (co-)homology of cotangent bundles has been computed before by Viterbo, Salamon -- Weber, and Abbondandolo -- Schwarz \cite{Viterbo_partI,
Salamon_Weber_Floer_homology_and_heat_flow,
Abbo_Schwarz_On_the_Floer_homology_of_cotangent_bundles} the following Theorem has been proved via the long exact sequence in \cite{Cieliebak_Frauenfelder_Oancea_Rabinowitz_Floer_homology_and_symplectic_homology}. An independent and direct proof has been given in \cite{Abbondandolo_Schwarz_Estimates_and_computations_in_Rabinowitz_Floer_homology}.

\begin{Thm}[\cite{Cieliebak_Frauenfelder_Oancea_Rabinowitz_Floer_homology_and_symplectic_homology,Abbondandolo_Schwarz_Estimates_and_computations_in_Rabinowitz_Floer_homology}]\label{thm:RFH_of_cotangent_bundle_is_homology_of_loop_space}
\beq
\RFHb_*(S^*_gB,T^*B)=\begin{cases}
\H^{-*+1}(\L_B)&\text{if }*<0\\
\H_{*}(\L_B)&\text{if }*>1\\
\end{cases}
\eeq
where $\L_B:=C^\infty(S^1,B)$ is the free loop space of $B$.
\end{Thm}

\begin{Rmk}
In the remaining degrees $*=0,1$ the answer is known and depends on the Euler class, see \cite{Cieliebak_Frauenfelder_Oancea_Rabinowitz_Floer_homology_and_symplectic_homology,Abbondandolo_Schwarz_Estimates_and_computations_in_Rabinowitz_Floer_homology}.
\end{Rmk}

\begin{Rmk}\label{rmk:merry_abbo_schwarz_iso}
In \cite{Merry_On_the_RFH_of_twisted_cotangent_bundles} Merry extends Theorem \ref{thm:RFH_of_cotangent_bundle_is_homology_of_loop_space} to energy hypersurfaces above the Ma\~n\'e critical value in twisted cotangent bundles.
\end{Rmk}

\section{Applications and Results}

\subsection{Symplectic and Contact topology}

As we pointed out in Theorem \ref{thm:RFH_does_not_depend_on_exterior} Rabinowitz Floer homology does not depend on the exterior. It is an open question to what extent Rabinowitz Floer homology $\RFH(\Sigma,M)$ depends on the filling $M$. A partial independence result is the following Theorem.

\begin{Thm}\label{thm:RFH_does_not_depend_on_filling}
Let $\dim B\geq4$ and $\pi_1(B)=\{1\}$ and let $(W,\om=d\lambda)$ be a compact exact symplectic manifold with $\p W \cong S^*_gB$. If $(\p W,\lambda|_{\p W})$ is contactomorphic to $(S^*_gB,\lambda_{std}|_{S^*_gB})$  then 
\beq
\RFHb_*(\p W,W)\cong\RFHb_*(S^*_gB,T^*B)\;.
\eeq
\end{Thm}
This is a special case of a result proved in \cite{Cieliebak_Frauenfelder_Oancea_Rabinowitz_Floer_homology_and_symplectic_homology}. The crucial ingredient in the above Theorem is that there exist no rigid finite energy planes in the filling $W$. If $B$ is a sphere then the above theorem can in fact be checked by a direct computation, see \cite{Cieliebak_Frauenfelder_Restrictions_to_displaceable_exact_contact_embeddings}.

\begin{Cor}
Under the assumptions of Theorem \ref{thm:RFH_does_not_depend_on_filling} $S^*_gB$ does not admit an exact contact embedding into $\R^{2n}$ or, more generally, into a subcritical Stein manifold.  
\end{Cor}

\begin{proof}
We assume by contradiction that there exists an exact contact embedding of $S^*_gB$ into a subcritical Stein manifold $(M,\om=d\lambda)$. We denote by $\Sigma$ the image of $S^*_gB$ in $M$. Since $\H_{2n-1}(M)=0$, $2n=\dim M$, the hypersurface $\Sigma$ bounds a compact region $W$ in $M$. Because any compact subset of a subcritical Stein manifold is displaceable, see Biran -- Cieliebak \cite{Biran_Cieliebak_Lagrangian_embeddings_into_subcritical_Stein_manifolds}, $\RFHb_*(\Sigma,M)\cong0$. On the other hand by Theorem \ref{thm:RFH_does_not_depend_on_exterior} and \ref{thm:RFH_does_not_depend_on_filling} we know $0\cong\RFHb_*(\Sigma,M)\cong\RFHb_*(\Sigma, W)\cong\RFHb_*(S^*_gB,T^*B)$. This contradicts Theorem \ref{thm:RFH_of_cotangent_bundle_is_homology_of_loop_space}.
\end{proof}

\subsection{Global perturbations of Hamiltonian systems}

We recall from section \ref{sec:perturbed_Rabinwitz_action_functional} that critical points of $\A^\Mp$, $\Mp=(F,H)$, give rise to leaf-wise intersections.

\begin{Thm}\label{thm:existence_of_LI}
Let $\Sigma\subset(M,\om=d\lambda)$ be a closed, bounding, restricted contact type hypersurface and $M$ convex at infinity or geometrically bounded. We denote by $\wp>0$ the smallest period of a Reeb orbit on $(\Sigma,\lambda|_\Sigma)$ which is contractible in $M$. Let $\psi\in\Ham_c(M,\om)$, that is, $\psi$ is a Hamiltonian diffeomorphism generated by a Hamiltonian with compact support. If the Hofer norm $||\psi||<\wp$ then there exists a leaf-wise intersection.
\end{Thm}

This was proved in \cite{Albers_Frauenfelder_Leafwise_intersections_and_RFH}. An alternative proof is given by G\"urel in \cite{Gurel_leafwise_coisotropic_intersection}. Kang \cite{Kang_Existence_of_leafwise_intersection_points_in_the_unrestricted_case} found an extension of Theorem \ref{thm:existence_of_LI} if $\Sigma$ is only of contact type. In \cite{Albers_Momin_Cup_length_estimates_for_leafwise_intersections} a cup-length estimate for leaf-wise intersections is proved under the assumptions of Theorem \ref{thm:existence_of_LI}.

\begin{Rmk}
The existence of displaceable hypersurfaces $\Sigma$ shows that a smallness assumption $||\psi||<\wp$ is necessary. 
\end{Rmk}

\begin{Rmk}
If $\psi=\id$ in Theorem \ref{thm:existence_of_LI} then each point in $\Sigma$ is a leaf-wise intersection point via constant critical points. The leaf-wise intersection point found in Theorem \ref{thm:existence_of_LI} arises as a perturbation of a constant critical point by a stretching-of-the-neck argument for gradient flow lines.  If this stretching-of-the-neck argument is replaced by local Floer homology around the constant critical points then for a generic Hamiltonian diffeomorphism $\psi$ in Theorem \ref{thm:existence_of_LI} there exists $\sum_{i=0}^{\dim\Sigma} b_i(\Sigma)$ different leaf-wise intersection points, see \cite{Albers_Frauenfelder_Leafwise_intersections_and_RFH,Kang_Existence_of_leafwise_intersection_points_in_the_unrestricted_case}.
\end{Rmk}

We recall that $\L_B=C^\infty(S^1,B)$ denotes the free loop space of $B$.

\begin{Thm}\label{thm:existence_of_infinitely_many_LI}
Let $\dim\H_*(\L_B)=\infty$. If $\dim B\geq2$ and $\Sigma\subset T^*B$ is a generic fiber-wise star-shaped hypersurface then for a generic $\psi\in\Ham_c(T^*B)$ there exist infinitely many leaf-wise intersection points. 
\end{Thm}

We point out that there is no assumption on the Hofer norm $||\psi||$ of $\psi$ in Theorem \ref{thm:existence_of_infinitely_many_LI}. This was proved in \cite{Albers_Frauenfelder_Leafwise_Intersections_Are_Generically_Morse} along the following lines. Since $\Sigma$ is fiber-wise star-shaped it is of restricted contact type and isotopic to $S^*_gB$ through restricted contact type hypersurfaces. Thus, by the assumption $\dim\H_*(\L_B)=\infty$, invariance of Rabinowitz Floer homology, and Theorem \ref{thm:RFH_of_cotangent_bundle_is_homology_of_loop_space} we conclude that the Rabinowitz Floer homology $\RFH(\Sigma,T^*B)$ is infinite dimensional. In particular, since for a generic perturbation the corresponding Rabinowitz action functional $\A^\Mp$ is Morse, $\A^\Mp$ has infinitely many critical points. Then a transversality result using $\dim B\geq2$ yields that generically critical points of $\A^\Mp$ won't lie on a closed leaf. Thus, the assertion of Theorem \ref{thm:existence_of_infinitely_many_LI} follows from Proposition \ref{prop:critical_points_give_LI}.

\begin{Rmk}
Theorems \ref{thm:existence_of_LI} and \ref{thm:existence_of_infinitely_many_LI} have been proved for energy hypersurfaces above Ma\~n\'e's critical value in twisted cotangent bundles in \cite{Merry_On_the_RFH_of_twisted_cotangent_bundles}. 
\end{Rmk}

\begin{Rmk}
As the sketch of the proof of Theorem \ref{thm:existence_of_infinitely_many_LI} shows to obtain generically infinitely many leaf-wise intersection points it is enough to prove that Rabinowitz Floer homology is infinite dimensional. A large class of examples with this property has been constructed by Albers -- McLean in \cite{Albers_McLean_SH_and_infinitley_many_LI}.
\end{Rmk}

Using spectral invariants in Rabinowitz Floer homology Theorem \ref{thm:existence_of_infinitely_many_LI} can be improved as follows, see \cite{Albers_Frauenfelder_Spectral_invariants_in_RFH}.

\begin{Thm}\label{thm:existence_of_infinitely_many_LI_with_spectral_invariants}
Let $\dim\H_*(\L_B)=\infty$. If $\dim B\geq2$ and $\Sigma\subset T^*B$ is a fiber-wise star-shaped hypersurface then for $\psi\in\Ham_c(T^*B)$ there exist infinitely many leaf-wise intersection points or a leaf-wise intersection point on a closed leaf. 
\end{Thm}

This Theorem is equivalent to the assertion that the Rabinowitz action functional always, even in degenerate situations, has infinitely many critical points. This does not follow from Theorem \ref{thm:existence_of_infinitely_many_LI} since for degenerate functions the Morse inequalities might fail.

The new ingredient in Theorem \ref{thm:existence_of_infinitely_many_LI_with_spectral_invariants} are spectral invariants. The idea behind spectral invariants is the following. To each homology class a critical value is associated. If the action functional is Morse this is done by a mini-max procedure. It can then be shown that this assignment is locally Lipschitz continuous in changes of the Moser pair. Thus, it extends to all Moser pairs. The crucial property of the extension is that it still assigns critical values to homology classes, even in degenerate situations.

To prove Theorem \ref{thm:existence_of_infinitely_many_LI_with_spectral_invariants} it is shown that the set of spectral values is unbounded and therefore gives rise to infinitely many critical points which are distinguished by their critical values.

It is an interesting open question whether a Gromoll-Meyer type theorem, \cite{Gromoll_Meyer_Periodic_geodesics_on_compact_riemannian_manifolds,Matthias_Verallgemeinerung_von_Gromoll_Meyer}, holds for leaf-wise intersection points, that is, if there exist infinitely many leaf-wise intersection points even in the case when there is a leaf-wise intersection point on a closed leaf. This problem is intimately related to the existence problem of (geometrically distinct) geodesics. Katok's example, \cite{Katok_Ergodic_perturbations_of_degenerate_integrable_Hamiltonian_systems} suggests that, as in the Gromoll-Meyer Theorem, a growth condition for the homology of the loop space is necessary. Interesting research in the direction can be found in \cite{Ginzburg_Gurel_Local_Floer_Homology_and_the_Action_Gap}.

Since $S^{2n-1}\subset\R^{2n}$ is displaceable for general Hamiltonian perturbations there need not exist leaf-wise intersection points. However, if the class of Hamiltonian perturbations is restricted to preserve symmetries this picture might change dramatically. Such a phenomenon was discovered by Ekeland -- Hofer in \cite{Ekeland_Hofer_Two_symplectic_fixed_point_theorems_with_applications_to_Hamiltonian_dynamics}. They prove that if $\Sigma$ is a centrally symmetric, restricted contact type hypersurface in $\R^{2n}$ and $\phi^t$ is an isotopy of centrally symmetric Hamiltonian perturbations there always exists a leaf-wise intersection point for $\phi^1$. If restricted contact type is replaced by star-shaped it is proved in \cite{Albers_Frauenfelder_Remark_on_a_Thm_by_Ekeland_Hofer} that under the same symmetry assumptions there exist infinitely many leaf-wise intersection points or a leaf-wise intersection point on a closed leaf. The proof uses a computation of $\Z/2$-invariant Rabinowitz Floer homology.

\subsection{Ma\~n\'e's critical value}

Let $(B,g)$ be a closed Riemannian manifold and $\sigma\in\Om^2(B)$ be a closed 2-form which vanishes on $\pi_2(B)$. Thus, on the universal cover $\pi:\widetilde{B}\to B$ the 2-form $\pi^*\sigma$ is exact. The twisted cotangent bundle is $(T^*B,\om:=\om_{std}+\tau^*\sigma)$ where $\tau:T^*B\to B$ is the projection. We fix a potential $U:B\to\R$ and define $F:T^*B\to\R$ by $F(q,p):=\tfrac12||p||_g^2+U(q)$. Then the Ma\~n\'e critical value is
\beq
c=c(g,\sigma,U):=\inf_{\theta}\;\sup_{q\in\widetilde{B}}\;\widetilde{F}(q,\theta_q)
\eeq
where $\widetilde{F}=F\circ\pi$ and $\theta\in\Om^1(\widetilde{B})$ satisfies $d\theta=\pi^*\sigma$.

In physical terms the Hamiltonian dynamics of $X_F$ taken with respect to the twisted symplectic form describes the motion of a particle on $B$ subject to a conservative force $-\nabla U$ and a magnetic field $\sigma$. For energies above the Ma\~n\'e critical value the magnetic field is interpreted as small compared to the energy. 

The hypersurfaces $\Sigma_k:=F^{-1}(k)$ above the Ma\~n\'e critical value, i.e.~$k>c$, are of virtual restricted contact type, see Remark \ref{rmk:virtual_restricted_contact_type}. If there are no topological obstructions $\Sigma_k$ is displaceable for sufficiently small $k$ if $\sigma\neq0$. This follows from the fact that the zero-section $B$ in $T^*B$ is displaceable, see Polterovich \cite{Polterovich_An_obstacle_to_non_Lagrangian_intersections}.

\begin{Thm}[\cite{Merry_On_the_RFH_of_twisted_cotangent_bundles}]
The hypersurface $\Sigma_k$ for $k>c$ is not displaceable. 
\end{Thm}

This follows from the fact that Rabinowitz Floer homology above the Ma\~n\'e critical value does not vanish, see Remark \ref{rmk:merry_abbo_schwarz_iso} .

At the Ma\~n\'e critical value the hypersurface $\Sigma_c$ is in general not stable or of virtual restricted contact type. It might even be singular, e.g.~if the magnetic field is zero. Moreover, there are examples where $\Sigma_c$ carries no closed characteristics, e.g.~the horocycle flow. The article \cite{Cieliebak_Frauenfelder_Paternain_Symplectic_Topology} proposes the following {\em paradigms}:

\begin{itemize}
\item[$(k>c)$] Above the Ma\~n\'e critical value, $\Sigma_k$ is virtually contact. It may or may not be stable. Its Rabinowitz Floer homology $\RFH(\Sigma_k)$ is defined and nonzero, so $\Sigma_k$ is non-displaceable (now proved by Merry \cite{Merry_On_the_RFH_of_twisted_cotangent_bundles}). The dynamics on $\Sigma_k$ is like that of a geodesic flow; in particular, it has a periodic orbit in every nontrivial free homotopy class.

\item[$(k=c)$] At the Ma\~n\'e critical value, $\Sigma_k$ is non-displaceable and can be expected to be non-stable (hence non-contact). 

\item[$(k<c)$] Below the Ma\~n\'e critical value, $\Sigma_k$ may or may not be of contact type. It is stable and displaceable (provided that $\chi(M)=0$), so its Rabinowitz Floer homology $\RFH(\Sigma_k)$ is defined and vanishes. In particular, $\Sigma_k$ has a contractible periodic orbit.
\end{itemize}

It is unknown whether $\Sigma_k$ is always stable below the Ma\~n\'e critical value. In \cite{Cieliebak_Frauenfelder_Paternain_Symplectic_Topology} many examples can be found supporting these paradigms.

\subsubsection*{Acknowledgments}
The authors thank Will Merry and Gabriel Paternain for various helpful comments on an earlier version.

This article was written during visits of the authors at the Institute for Advanced Study, Princeton. The authors thank the Institute for Advanced Study for their stimulating working atmospheres. 

This material is based upon work supported by the National Science Foundation under agreement No.~DMS-0635607 and DMS-0903856. Any opinions, findings and conclusions or recommendations expressed in this material are those of the authors and do not necessarily reflect the views of the National Science Foundation.

%
\bibliographystyle{amsalpha}
\bibliography{../../../Bibtex/bibtex_paper_list}

\providecommand{\bysame}{\leavevmode\hbox to3em{\hrulefill}\thinspace}
\providecommand{\MR}{\relax\ifhmode\unskip\space\fi MR }
\providecommand{\MRhref}[2]{%
  \href{http://www.ams.org/mathscinet-getitem?mr=#1}{#2}
}
\providecommand{\href}[2]{#2}
\begin{thebibliography}{CFP09b}

\bibitem[AF08]{Albers_Frauenfelder_Leafwise_Intersections_Are_Generically_Mors%
e}
P.~Albers and U.~Frauenfelder, \emph{{Infinitely many leaf-wise intersections
  on cotangent bundles}}, 2008, arXiv:0812.4426.

\bibitem[AF10a]{Albers_Frauenfelder_Leafwise_intersections_and_RFH}
\bysame, \emph{{Leaf-wise intersections and Rabinowitz Floer homology}}, J.
  Topol. Anal. \textbf{2} (2010), no.~1, 77--98.

\bibitem[AF10b]{Albers_Frauenfelder_Remark_on_a_Thm_by_Ekeland_Hofer}
\bysame, \emph{{On a Theorem by Ekeland-Hofer}}, 2010, arXiv:1001.3386, to
  appear in Israel Journal of Mathematics.

\bibitem[AF10c]{Albers_Frauenfelder_Spectral_invariants_in_RFH}
\bysame, \emph{{Spectral invariants in Rabinowitz Floer homology and global
  Hamiltonian perturbations}}, 2010, arXiv:1001.2920.

\bibitem[AM09]{Albers_McLean_SH_and_infinitley_many_LI}
P.~Albers and M.~McLean, \emph{{Non-displaceable contact embeddings and
  infinitely many leaf-wise intersections}}, 2009, arXiv:0904.3564, to appear
  in Journal of Symplectic Geometry.

\bibitem[AM10]{Albers_Momin_Cup_length_estimates_for_leafwise_intersections}
P.~Albers and A.~Momin, \emph{{Cup-length estimates for leaf-wise
  intersections}}, 2010, arXiv:1002.3283.

\bibitem[AS06]{Abbo_Schwarz_On_the_Floer_homology_of_cotangent_bundles}
A.~Abbondandolo and M.~Schwarz, \emph{{On the Floer homology of cotangent
  bundles}}, Comm. Pure Appl. Math. \textbf{59} (2006), no.~2, 254--316.

\bibitem[AS09]{Abbondandolo_Schwarz_Estimates_and_computations_in_Rabinowitz_F%
loer_homology}
\bysame, \emph{{Estimates and computations in Rabinowitz-Floer homology}}, J.
  Topol. Anal. \textbf{1} (2009), no.~4, 307--405.

\bibitem[Ban80]{Banyaga_On_fixed_points_of_symplectic_maps}
A.~Banyaga, \emph{On fixed points of symplectic maps}, Invent. Math.
  \textbf{56} (1980), no.~3, 215--229.

\bibitem[BC02]{Biran_Cieliebak_Lagrangian_embeddings_into_subcritical_Stein_ma%
nifolds}
P.~Biran and K.~Cieliebak, \emph{Lagrangian embeddings into subcritical {S}tein
  manifolds}, Israel J. Math. \textbf{127} (2002), 221--244.

\bibitem[CF09a]{Cieliebak_Frauenfelder_Restrictions_to_displaceable_exact_cont%
act_embeddings}
K.~Cieliebak and U.~Frauenfelder, \emph{{A Floer homology for exact contact
  embeddings}}, Pacific J. Math. \textbf{293} (2009), no.~2, 251--316.

\bibitem[CF09b]{Cieliebak_Frauenfelder_Morse_homology_on_noncompact_manifolds}
\bysame, \emph{{Morse homology on noncompact manifolds}}, 2009,
  arXiv:0911.1805.

\bibitem[CFH95]{Cieliebak_Floer_Hofer_Symplectic_homology_II_A_general_constru%
ction}
K.~Cieliebak, A.~Floer, and H.~Hofer, \emph{Symplectic homology. {II}. {A}
  general construction}, Math. Z. \textbf{218} (1995), no.~1, 103--122.

\bibitem[CFO09]{Cieliebak_Frauenfelder_Oancea_Rabinowitz_Floer_homology_and_sy%
mplectic_homology}
K.~Cieliebak, U.~Frauenfelder, and A.~Oancea, \emph{{Rabinowitz Floer homology
  and symplectic homology}}, 2009, arXiv:0903.0768, to appear in Annales
  Scientifiques de L'ENS.

\bibitem[CFP09a]{Cieliebak_Frauenfelder_Paternain_Stability_is_not_open}
K.~Cieliebak, U.~Frauenfelder, and G.~Paternain, \emph{{Stability is not
  open}}, 2009, arXiv:0908.2540.

\bibitem[CFP09b]{Cieliebak_Frauenfelder_Paternain_Symplectic_Topology}
\bysame, \emph{{Symplectic Topology of Ma\~n\'e's critical value}}, 2009,
  arXiv:0903.0700.

\bibitem[CGS00]{Cieliebak_Gaio_Salamon_J_hol_curves_moment_maps_and_invariants%
_of_Hamiltonian_group_actions}
K.~Cieliebak, A.~R. Gaio, and D.~A. Salamon, \emph{{$J$}-holomorphic curves,
  moment maps, and invariants of {H}amiltonian group actions}, Internat. Math.
  Res. Notices (2000), no.~16, 831--882.

\bibitem[CV10]{Cieliebak_Volkov_First_steps_in_stable_Hamiltonian_topology}
K.~Cieliebak and E.~Volkov, \emph{{First steps in stable Hamiltonian
  topology}}, 2010, arXiv:1003.5084.

\bibitem[Dra08]{Dragnev_Symplectic_rigidity_symplectic_fixed_points_and_global%
_perturbations_of_Hamiltonian_systems}
D.~L. Dragnev, \emph{Symplectic rigidity, symplectic fixed points, and global
  perturbations of {H}amiltonian systems}, Comm. Pure Appl. Math. \textbf{61}
  (2008), no.~3, 346--370.

\bibitem[EH89]{Ekeland_Hofer_Two_symplectic_fixed_point_theorems_with_applicat%
ions_to_Hamiltonian_dynamics}
I.~Ekeland and H.~Hofer, \emph{Two symplectic fixed-point theorems with
  applications to {H}amiltonian dynamics}, J. Math. Pures Appl. (9) \textbf{68}
  (1989), no.~4, 467--489 (1990).

\bibitem[Fra04]{Frauenfelder_Arnold_Givental_Conjecture}
U.~Frauenfelder, \emph{The {A}rnold-{G}ivental conjecture and moment {F}loer
  homology}, Int. Math. Res. Not. (2004), no.~42, 2179--2269.

\bibitem[GG03]{Ginzburg_Gurel_A_C_2_smooth_counterexample_to_the_Hamiltonian_S%
eifert_conjecture}
V.~L. Ginzburg and B.~G{\"u}rel, \emph{A {$C^2$}-smooth counterexample to the
  {H}amiltonian {S}eifert conjecture in {$\Bbb R^4$}}, Ann. of Math. (2)
  \textbf{158} (2003), no.~3, 953--976.

\bibitem[GG07]{Ginzburg_Gurel_Local_Floer_Homology_and_the_Action_Gap}
\bysame, \emph{{Local Floer Homology and the Action Gap}}, 2007,
  arXiv:0709.4077, to appear in Journal of Symplectic Geometry.

\bibitem[Gin07]{Ginzburg_Coisotropic_intersections}
V.~L. Ginzburg, \emph{Coisotropic intersections}, Duke Math. J. \textbf{140}
  (2007), no.~1, 111--163.

\bibitem[GM69]{Gromoll_Meyer_Periodic_geodesics_on_compact_riemannian_manifold%
s}
D.~Gromoll and W.~Meyer, \emph{Periodic geodesics on compact {R}iemannian
  manifolds}, J. Differential Geometry \textbf{3} (1969), 493--510.

\bibitem[G{\"u}r09]{Gurel_leafwise_coisotropic_intersection}
B.~G{\"u}rel, \emph{{Leafwise Coisotropic Intersections}}, Int. Math. Res. Not.
  (2009), article ID rnp 164.

\bibitem[Hof90]{Hofer_On_the_topological_properties_of_symplectic_maps}
H.~Hofer, \emph{On the topological properties of symplectic maps}, Proc. Roy.
  Soc. Edinburgh Sect. A \textbf{115} (1990), no.~1-2, 25--38.

\bibitem[HS95]{Hofer_Salamon_Floer_homology_and_Novikov_rings}
H.~Hofer and D.~A. Salamon, \emph{Floer homology and {N}ovikov rings}, The
  Floer memorial volume, Progr. Math., vol. 133, Birkh\"auser, Basel, 1995,
  pp.~483--524.

\bibitem[Kan09]{Kang_Existence_of_leafwise_intersection_points_in_the_unrestri%
cted_case}
J.~Kang, \emph{{Existence of leafwise intersection points in the unrestricted
  case}}, 2009, arXiv:0910.2369.

\bibitem[Kat73]{Katok_Ergodic_perturbations_of_degenerate_integrable_Hamiltoni%
an_systems}
A.~B. Katok, \emph{Ergodic perturbations of degenerate integrable {H}amiltonian
  systems}, Izv. Akad. Nauk SSSR Ser. Mat. \textbf{37} (1973), 539--576.

\bibitem[Mat80]{Matthias_Verallgemeinerung_von_Gromoll_Meyer}
H.~Matthias, \emph{Zwei {V}erallgemeinerungen eines {S}atzes von {G}romoll und
  {M}eyer}, Bonner Mathematische Schriften 126, Universit\"at Bonn
  Mathematisches Institut, 1980.

\bibitem[Mer10]{Merry_On_the_RFH_of_twisted_cotangent_bundles}
W.~Merry, \emph{{On the Rabinowitz Floer homology of twisted cotangent
  bundles}}, 2010, arXiv:1002.0162.

\bibitem[Mos78]{Moser_A_fixed_point_theorem_in_symplectic_geometry}
J.~Moser, \emph{A fixed point theorem in symplectic geometry}, Acta Math.
  \textbf{141} (1978), no.~1--2, 17--34.

\bibitem[MP09]{Macarini_Paternain_On_the_stability_of_Mane_critical_hypersurfa%
ces}
L.~Macarini and G.~Paternain, \emph{{On the stability of Ma\~n\'e critical
  hypersurfaces}}, 2009, arXiv:0910.5728.

\bibitem[Pol95]{Polterovich_An_obstacle_to_non_Lagrangian_intersections}
L.~Polterovich, \emph{An obstacle to non-{L}agrangian intersections}, The
  {F}loer memorial volume, Progr. Math., vol. 133, Birkh\"auser, Basel, 1995,
  pp.~575--586.

\bibitem[Rab78]{Rabinowitz_Periodic_solutions_of_Hamiltonian_systems}
P.~H. Rabinowitz, \emph{Periodic solutions of {H}amiltonian systems}, Comm.
  Pure Appl. Math. \textbf{31} (1978), no.~2, 157--184.

\bibitem[Rab79]{Rabinowitz_Periodic_solutions_of_Hamiltonian_systems_on_a_pres%
cribed_enery_surface}
\bysame, \emph{Periodic solutions of a {H}amiltonian system on a prescribed
  energy surface}, J. Differential Equations \textbf{33} (1979), no.~3,
  336--352.

\bibitem[RS96]{Robbin_Salamon_Feynman_path_integrals}
J.~Robbin and D.~A. Salamon, \emph{Feynman path integrals on phase space and
  the metaplectic representation}, Math. Z. \textbf{221} (1996), no.~2,
  307--335.

\bibitem[Sch06]{Schlenk_Applications_of_Hofers_geometry_to_Hamiltonian_dynamic%
s}
F.~Schlenk, \emph{Applications of {H}ofer's geometry to {H}amiltonian
  dynamics}, Comment. Math. Helv. \textbf{81} (2006), no.~1, 105--121.

\bibitem[SW06]{Salamon_Weber_Floer_homology_and_heat_flow}
D.~A. Salamon and J.~Weber, \emph{{Floer homology and the heat flow}}, Geom.
  Funct. Anal. \textbf{16} (2006), no.~5, 1050--1138.

\bibitem[Vit99]{Viterbo_partI}
C.~Viterbo, \emph{{Functors and computations in {F}loer homology with
  applications. {I}}}, Geom. Funct. Anal. \textbf{9} (1999), no.~5, 985--1033.

\bibitem[Wei79]{Weinstein_The_conjecture}
A.~Weinstein, \emph{On the hypotheses of {R}abinowitz' periodic orbit
  theorems}, J. Differential Equations \textbf{33} (1979), no.~3, 353--358.

\bibitem[Zil08]{Ziltener_coisotropic}
F.~Ziltener, \emph{{Coisotropic Submanifolds, Leafwise Fixed Points, and
  Presymplectic Embeddings}}, 2008, arXiv:0811.3715, to appear in Journal of
  Symplectic Geometry.

\end{thebibliography}
\end{document}